\documentclass[10pt]{article}
\usepackage{amsmath,amssymb,amsthm, epsfig}
\usepackage{amsfonts}
\usepackage{amsmath}
\usepackage{amssymb}
\usepackage{color}
\usepackage[mathscr]{eucal}

\title{Regularity for quasilinear equations \\ on degenerate singular sets}

\author{by  \vspace{0.3cm}  \\    \textsc{eduardo v. teixeira} \\ \textit{\footnotesize Universidade Federal do Cear\'a}  \\ \textit{\footnotesize Fortaleza, CE, Brazil} }

\date{}

\newlength{\hchng}
\newlength{\vchng}
\def \div {\mathrm{div}}
\def \dist {\mathrm{dist}}

\def \suchthat {: }

\def \S {\mathscr{S}}

\newtheorem{theorem}{Theorem}[section]
\newtheorem{lemma}[theorem]{Lemma}
\newtheorem{proposition}[theorem]{Proposition}
\newtheorem{corollary}[theorem]{Corollary}

\theoremstyle{definition}

\theoremstyle{remark}

\numberwithin{equation}{section}

\newcommand{\intav}[1]{\mathchoice {\mathop{\vrule width 6pt height 3 pt depth  -2.5pt
\kern -8pt \intop}\nolimits_{\kern -6pt#1}} {\mathop{\vrule width
5pt height 3  pt depth -2.6pt \kern -6pt \intop}\nolimits_{#1}}
{\mathop{\vrule width 5pt height 3 pt depth -2.6pt \kern -6pt
\intop}\nolimits_{#1}} {\mathop{\vrule width 5pt height 3 pt depth
-2.6pt \kern -6pt \intop}\nolimits_{#1}}}

\begin{document}
\maketitle

\begin{abstract}
We prove a new, universal gradient continuity estimate for solutions to quasilinear equations with varying coefficients at points on its critical singular set of degeneracy  $\S(u)  := \{X \suchthat D u(X) = 0 \}$.  Our main Theorem reveals that along $\S(u)$,  $u$ is asymptotically as regular as solutions to constant coefficient equations. In particular, along the critical set $\S(u)$, $Du$ enjoys a modulus of continuity much superior than the, possibly low, continuity feature of the coefficients. The results are new even in the context of linear elliptic equations, where it is herein shown that $H^1$-weak solutions to $\text{div}\left (a_{ij}(X)Du \right ) = 0$, with $a_{ij}$ elliptic and Dini-continuous are actually $C^{1,1^{-}}$ along $\S(u)$. The results and insights of this work foster a new understanding on smoothness properties of solutions to degenerate or singular equations, beyond typical elliptic regularity estimates, precisely where the diffusion attributes of the equation collapse.  \medskip

\noindent \textit{MSC2010:} 35J70, 35J75, 35J62, 35B65.

\medskip

\noindent \textbf{Keywords:} Degenerate and singular elliptic equations, quasilinear problems, regularity theory.

%\tableofcontents 
\end{abstract}

%\tableofcontents

\section{Introduction}

In this paper we investigate local continuity behavior of the gradient of solutions, along the zero gradient patch, to singular or degenerate elliptic equations with varying coefficients: 
	\begin{equation} \label{Eq Introd}
 	-\div \left ( a(X, Du) \right ) = \mu,
\end{equation}
in a domain $\Omega \subset \mathbb{R}^n$, $n\ge 2$, where $\mu$ is a source function or a measure with finite total mass, $|\mu|(\Omega) <+\infty$.  The model for the class of equations we treat in this article is the non-homogeneous $p$-Laplacean equation with varying coefficients:
\begin{equation} \label{Eq p-Lap}
 	-\div \left ( \varsigma(X) |Du|^{p-2} Du \right ) = \mu,
\end{equation}
where $\varsigma$ is bounded away from $0$ and $\infty$ and satisfies a middle continuity assumption to be specified in Section \ref{section set-up}. Equation \eqref{Eq p-Lap}, or more generally Equation \eqref{Eq Introd} appear in several contexts, as they represent an anisotropic, quasilinear law for diffusion. 

\par
\medskip

Existence and fine regularity properties of solutions to Equation \eqref{Eq p-Lap} have been subject of massive study through the past half century, or so. Its mathematical analysis is rather more involved than its linear counterpart ($p=2$), mainly due to its singular or degenerate behavior along the zero gradient patch, the so called singular set of the solution $u$:
\begin{equation}\label{sing set}
	\S(u) := \left \{ X \in \Omega : Du(X) = 0  \right \}.
\end{equation}
At a point $X_0 \in \S$, the coefficients of equation \eqref{Eq p-Lap} either blow-up, in the case $1<p < 2$, or  else degenerate for $p>2$. Such features impel less efficient smoothing effects of the diffusion of the operator and the regularity theory for weak solutions to equation \eqref{Eq p-Lap} becomes a rather challenging mathematical issue. 

\par
\medskip

The first major result in the area is due to Ural�tseva, who proved in \cite{U}, for the degenerate case, $p\ge 2$, that $p$-harmonic functions, i.e., solutions to the homogeneous, constant coefficient equation,
\begin{equation}\label{p-harmonic}
	-\Delta_p u := - \div \left (|Du|^{p-2} Du \right ) = 0,
\end{equation} 
are locally of class $C^{1,\alpha_p}$ for some exponent $0 < \alpha_p < 1$. Uhlenbeck, in \cite{Uhl}, provided further extensions. Similar estimate  for the singular case, $1 < p < 2$, was established in \cite{DiB} and \cite{Lewis}. At this point is it interesting to notice that, alway from the singular set, $\S(u)$, $p$-harmonic functions are in fact quite smooth - real analytic. Such conclusion follows by standard elliptic regularity theory. Nevertheless, $C_\text{loc}^{1,\alpha_p}$  is indeed optimal, since along its singular set $\S(u)$, $p$-harmonic functions are not, in general, of class $C^2$, nor even $C^{1,1}$.

\par
\medskip

Regularity theory for varying coefficient equations is even more involved as continuity features of the coefficients $\varsigma(X)$ restrict even more the smoothing properties of the operator. The corresponding linear theory, $p=2$, goes back to the classical Schauder's {\it a priori} estimates, which state that solutions to 
$$
	-\div (a_{ij}(X) Du ) = 0, \quad \lambda \text{Id} \le a_{ij} \le \Lambda \text{Id}, ~ a_{ij} \in C^{0, \alpha},
$$
are locally of class $C^{1,\alpha}$. Such a result is optimal in several ways. Clearly if $a_{ij} \in C^{0, \alpha}$ and no better than that, one should not expect solutions to be smoother than $C^{1,\alpha}$, as simple 1d calculations show.  Also, it has been shown, see \cite{JMS}, that continuity of $a_{ij}$ is not enough to assure gradient bounds for solutions. 

\par
\medskip

The regularity theory for general nonlinear, varying coefficient equations, as in \eqref{Eq Introd} has become accessible just quite recently, through a rather sophisticated and powerful nonlinear potential theory, see \cite{DM, KM01, M01, M02, M03}. The   ultimate 	
scientific endowment of these recent works is a complete, essentially sharp, regularity theory for quasilinear equations as in \eqref{Eq Introd},  in terms of nonlinear potential properties of the datum $\mu$ (the so called nonlinear Wolff potential) and appropriate continuity of the coefficients, $X \mapsto a(X, \cdot)$. In vernacular terms (see Section \ref{section set-up} for precise discussion),  if the vector field $a$ has $C^{0, \epsilon}$ coefficients, $p\ge 2$, then 
\begin{equation} \label{reg DM}
	-\div \left (a(X, Du) \right ) = 0, \quad \text{implies} \quad h \in C_\text{loc}^{1, \beta}, \ \forall \beta<\frac{2\epsilon}{p},
\end{equation}
see \cite{KM01}, Theorem 1.4. It is worth noticing that such an estimate is asymptotically optimal, as explicit examples show.  

\par
\medskip

Yet in vernacular terms, the main, key result we prove in this present article endorses that as long as the coefficients of the equation are continuous enough as to assure local, {\it a priori} $C^1$ estimates, then on the degenerate singular set, $\S(u)$ -- aways regarded as the villain of the theory -- $u$ is asymptotically as regular as solutions to the constant coefficient equation. In particular at a singular point $X_0 \in \S(u)$, $Du$ has, in general, a much stronger modulus of continuity than the one confined by the coefficients of the equation, as in estimate \eqref{reg DM}, for $\epsilon \ll 1$.  

Even through the prism of the classical linear Schauder theory, the last paragraph should not, in principle, be read without some dose of perplexity.  Indeed, if $u$ is a solution to a uniformly elliptic, divergence form equation 
$$
	-\div (a_{ij}(X) Du) = 0, \quad \text{ with say, } \quad a_{ij} \in C^{0, \frac{1}{1000}},
$$
then, as mentioned above, we have known that $u$ is locally of class $C^{1,\frac{1}{1000}}$, and this regularity is optimal. However, a consequence of our main Theorem, see Corollary \ref{cor p=2}, is that at singular point, $X_0 \in \S(u)$, in fact $u$ is much smoother, for instance, 
$$
	 u \in C^{1, \frac{999}{1000}} \text{ at any point } X_0 \in  \S(u).
$$
That is a much stronger smoothness property than the one granted by the regular theory. 

\par
\medskip

We have postponed the  assumptions and precise statements of the results of this article to Section \ref{section main theorems}. For those who have read the last two paragraphs with some consternation, we concluded this Introduction explaining the heuristics that conducted us towards this pool of results. Initially, it is elucidative to comprehend that even though elliptic equations in divergence form are of 2nd order, in fact it reflects an oscillation balance around constants, rather than affine functions, as in the non-divergence theory. That explains, to some extent, why in many situations, the divergence form regularity theory {\it has one derivative less} than the non-divergence one. Now, if we want to show that an arbitrary  given function $u$ is of class $C^{1,\alpha}$ at, say, the origin, our task is to find an affine function $\ell(X)$ that approximates $u$ up to an error of order $\text{O}(r^{1+\alpha})$. If $0$ happens to be a singular point for $u$, i.e., $0\in \S(u)$, then the 1st order of the approximation $\ell$ should be zero, and we are led to control the oscillation balance of $u$ around a real constant. Say, if $u(0) = |Du(0)| = 0$, then proving $u \in C^{1,\alpha}$ at the origin reduces to verifying $|u(X)| = \text{O}(r^{1+\alpha})$. The contrapositive of the last logic assertion is that there exists a sequence of points $X_j \to 0$ for which $|X_j|^{1+\alpha} |u(X_j)|^{-1} = \text{o}(1)$. Expanding and normalizing this sequence appropriately gives a sequence $u_j$ that converges to an entire function $u_\infty$,  solution to a homogeneous, constant coefficient equation. The limiting function $u_\infty$ should grow no more than $|X|^{1+\alpha}$, otherwise by a discrete iterative procedure, we could conclude the aimed $C^{1,\alpha}$ estimate. However, since $a$-harmonic functions are locally of class $C^{1,\alpha_M}$, for a maximal exponent $\alpha_M$ strictly bigger than $\alpha$, we would conclude $u_\infty \equiv 0$, leading us to a contradiction on a lower bound of the $L^\infty(\partial B_1)$ norm of the approximating functions $u_j$. All these heuristic reasonings, or else alternative corresponding steps, will be made precise along the remaining Sections of this manuscript, ultimately providing a conclusive proof of the main Theorem. 

\par
\medskip
 
The paper is organized as follows: in Section \ref{section set-up} we gather few tools and known results that support both the statements and the proofs of our main results. In Section \ref{section main theorems} we present the Theorems we show in this paper and comment on some implications they have on the current literature. The proof of the most general result, Theorem \ref{main NH},  is developed through the remaining Sections \ref{section flat approx}, \ref{section flat improv}, \ref{section red} and \ref{section final proof}.

\section{Preliminaries and some known tools} \label{section set-up}

In this Section we will explain the mathematical set-up involved in the paper. We will also gather some Theorems and tools that support the underlying theory behind our results.

Throughout this paper we shall assume the following standard structural assumptions
on the vector field $a \colon \Omega \times \mathbb{R}^n \to \mathbb{R}$:
\begin{equation}\label{Hyp a}
    \left \{
        \begin{array}{rll}
            |a(X,\xi)| + |\partial_\xi a(X, \xi) | \cdot \xi| &\le&  \Lambda   |\xi|^{p-1}  \\
            \lambda   |\xi_1|^{p-2} \cdot |\xi_2|^2 &\le & \langle \partial_\xi a(X, \xi_1)\xi_2, \xi_2 \rangle \\
            |a(X, \xi) - a(Y, \xi)| &\le& \tilde{\Lambda} \omega(|X-Y|) \cdot  |\xi|^{p-1},
        \end{array}
    \right.
\end{equation}
for $2 - \frac{1}{n} < p < n$,  positive constants $0 < \lambda \le \Lambda < +\infty$ and $\tilde{\Lambda} \ge 1$. As usual in the literature, we could also include a parameter  $s \ge 0$ as to distinghish degenerate/singular equations to non-degenerate/non-singular ones. Per our primary motivation, we have chosen to work only in the genuine  degenerate/singular situation, $s=0$.

Initially, we recall that for constant coefficient equations, i.e., $\omega \equiv 0$, hereafter written simply as
\begin{equation}\label{a-harmonic}
	-\div\left ( a(Dh) \right ) = 0,
\end{equation}
it is well established,  that an $a$-harmonic function is locally of class $C^{1,\alpha_M}$, for some maximal exponent $0 < \alpha_M < 1$, that depends only upon $n$, $p$, $\lambda$ and $\Lambda$, see for instance \cite{DiB}. The precise value of $\alpha_M$ has been determined, for the model equation, $-\Delta_p u = 0$,  in the two-dimensional space, \cite{IM}.  Hereafter in this paper $\alpha_M$ denotes the maximal H\"older exponent of the gradient of an $a$-harmonic function, i.e., solution to \eqref{a-harmonic}. The hole of $\alpha_M$ in the regualrity theory for general equations satisfying  the structural assumption \eqref{Hyp a}
\begin{equation}\label{a-eq}
	-\div\left ( a(X, Dv) \right ) = \mu
\end{equation}
is clear. Since no solution to the non-homogeneous, varying coefficients equation 
should be expected to be more regular than $a$-harmonic functions,  the maximal exponent $\alpha_M$ is naturally an asymptotically upper barrier for any $C^{1,\alpha}$ regularity theory for equation \eqref{a-eq}.  That is, any universal $C^{1,\alpha}$ regularity estimate for solutions to \eqref{a-eq}  requires  $\alpha< \alpha_M$.

For equations with varying coefficients, the function $\omega$ represents a given modulus of continuity for the coefficients of the operator $a$. Hereafter in this paper, we shall assume the following Dini-type condition on the coefficients:
\begin{equation} \label{Dini}
		\begin{array}{llll}
			\displaystyle \int_0^R \omega(\tau)^{\frac{2}{p}} \dfrac{d\tau}{\tau} &<& +\infty,   &\text{ in the case } p\ge 2 \\
			\displaystyle \int_0^R \omega(\tau)^{1-\sigma} \dfrac{d\tau}{\tau} &<& +\infty,    &\text{ for some } \sigma > 0, \text{ in the case } 2- \frac{1}{n} < p < 2.
		\end{array}
\end{equation}
Notice that condition \eqref{Dini} is immediately satisfies for equations with $C^{0,\epsilon}$ coefficients, i.e. for $\omega(t) = t^\epsilon$. Such assumption, in each regime, namely $p\ge 2$ or $2 - \frac{2}{n} < p \le 2$, is essentially optimal for $C^1$ estimates to equations with varying coefficients, see Theorem \ref{C1 est} below. 

It is also well established that  no gradient control can be obtained for solutions to non-homogeneous equations, unless it is enforced a minimal integrability condition on the source $\mu$ of order $\sim L^{n^{+}}$. To be more precise, throughout this paper we shall work under the assumption,
\begin{equation} \label{measure cond}
	\mu(X) \in L^{q}(\Omega), \quad q > n.
\end{equation}
Such condition could be further relaxed; however we have chosen to state our results based on condition \eqref{measure cond} to keep the presentation cleaner.  We can now state the ultimate, main $C^1$ regularity estimate available for non-homogeneous equations with varying coefficients. 

\begin{theorem}[\cite{DM-CV}, Theorem 4 and \cite{KM01}, Theorem 1.6] \label{C1 est} Let $u \in W^{1,p}(\Omega)$ be a weak solution to Equation \eqref{Eq Introd}, where $a$ satisfies \eqref{Hyp a} and \eqref{Dini}. Assume further that $\mu$ satisfies \eqref{measure cond}. Then $u \in C^1_\text{loc}(\Omega)$. In additional, for any subdomain $\Omega' \Subset \Omega$, there exists a universal modulus of continuity, $\tau$, depending only on $\Omega'$, $\Omega$, $\|u\|_{L^p(\Omega)}$, $\lambda, \Lambda, \tilde{\Lambda}, \omega$ and $\|\mu\|_{L^q(\Omega)}$, such that
$$
	\left | Du(X) - Du(Y) \right | \le \tau(|X-Y|), \quad \forall X, Y \in \Omega'.
$$
\end{theorem}

It is also interesting to read \cite{Manf1, Manf2}, for earlier results in this line. As mentioned earlier in the Introduction, Theorem \ref{C1 est} is a result of a powerful nonlinear potential theory that has been developed and shaped up since the groundbreaking work of Kilpel\"ainen and Mal\'y, \cite{KM}. For our purposes, Theorem \ref{C1 est} is the starting point of the results we shall prove in this current article, which shall be explained within the next Section.

\section{Main results and their consequences}\label{section main theorems}

In this Section we present, in rigorous forms, the main results we establish in this work.  We  also comment on few implications they have and indicate some potential consequences in the current literature. Let us start off by a less general, but rather emblematic Theorem, that follows as an offspring result from the general estimate we shall prove. 

\begin{theorem}\label{main} Let $u \in W^{1,p}(\Omega)$ be a weak solution to 
\begin{equation} \label{Eq-0}
 	-\div \left (a(X, D u) \right ) = 0,
\end{equation}
where $a$ satisfies the structural conditions \eqref{Hyp a} and \eqref{Dini}. Let $X_0$ be an arbitrary point at the singular set of $u$:
$$
	\S(u) := \left \{ Y \in \Omega \suchthat  Du(Y) = 0 \right \}.
$$
Then $u \in C^{1,\alpha_M^{-}}$ at $X_0$. More precisely, given any $0 < \alpha < \alpha_M$, there exists a constant $C>0$ depending only upon $\|u\|_{L^p(\Omega)}$,  $\dist(X_0, \partial \Omega)$, $n, p, \lambda, \Lambda,  \tilde{\Lambda}, \omega,$ and $\alpha$,  such that
$$
	\sup\limits_{Y\in B_r(X_0)} \left | u(Y) -  u(X_0)  \right | \le C r^{1 + \alpha}.
$$
\end{theorem}

Theorem \ref{main} revels a surprising gain of smoothness of $u$, beyond the continuity of the coefficients,  precisely along the singular set of the equation: the most delicate region to be analyzed. The exactly same regularity result is assured to solutions with bounded sources, i.e.,
$$
	-\div \left (a(X, D u) \right ) = f(X) \in L^\infty(\Omega),
$$
or more generally, for equations with BMO datum. When we project Theorem \ref{main} to the non-degenerate, linear regime, $p=2$, it provides a remarkable estimate which, to the best of our knowledge, is also new.

\begin{corollary}\label{cor p=2}  Let $u\in H^1(\Omega)$ be a weak solution to
$$
	-\div (a_{ij}(X) Du) = 0,
$$
where $a_{ij}$ is a positive definite, Dini-continuous matrix. Then, at any gradient zero point, i.e., $D u(Z) = 0$, $u$ is of class $C^{1,\alpha}$ for all $0< \alpha < 1$. 
\end{corollary}

In fact Corollary \ref{cor p=2} follows from Theorem \ref{main} since harmonic functions, $\Delta h = 0$, are of class $C^{1,1}$.   These results follow as a consequence of a more general, non-homogeneous result we shall prove in this work:

\begin{theorem}\label{main NH} Let $u \in W^{1,p}(\Omega)$ be a weak solution to 
\begin{equation} \label{Eq}
 	-\div \left (a(X, D u) \right ) = \mu
\end{equation}
where $a$ satisfies the structural conditions \eqref{Hyp a}, \eqref{Dini} and $\mu$ satisfies \eqref{measure cond}. Let $X_0$ be an arbitrary point at the singular set of $u$:
$$
	\S(u) := \left \{ Y \in \Omega \suchthat  Du(Y) = 0 \right \}.
$$
Then 
$$
	u \in C^{1,\min\{\alpha_M^{-},  \frac{q-n}{(p-1)q} \}}
$$ 
at $X_0$. That is, given $\alpha \in (0, \alpha_M) \cap (0, \frac{q-n}{(p-1)q}  ]$, there exists a constant $C>0$ depending only upon $\|u\|_{L^p(\Omega)}$, $\dist(X_0, \partial \Omega)$, $n, p, \lambda, \Lambda,  \tilde{\Lambda}, \omega,$ $\|\mu\|_{L^q(\Omega)}$ and $\alpha$,  such that
$$
	\sup\limits_{Y\in B_r(X_0)} \left | u(Y) -  u(X_0)  \right | \le C   r^{1 + \alpha}.
$$
\end{theorem}

\medskip

It is interesting to verify that   $\lim\limits_{q\to \infty} \frac{q-n}{(p-1)q} = \frac{1}{p-1}.$  In the next Sections we shall deliver a proof of Theorem \ref{main NH}. The strategy of the proof is based on a compactness approach and it is inspired by the revolutionary work of Luis Caffarelli on $W^{2,p}$ estimates for fully nonlinear equations, \cite{C1}. See also \cite{Teix} for similar reasoning on sharp regularity theory for quasilinear Poisson equations.
%%%%%%%%%%
%%%%%%%%%%
%%%%%%%%%%

\section{Singular approximation} \label{section flat approx}

In this Section we prove the first main ingredient we need  in the proof of Theorem \ref{main NH}. It accounts a refined singular, flat approximation result that assures that, if the right-hand-side of the equation, $\mu$, is close to zero, the coefficients are close to constant and $u$ is flat enough at $0$, then it is possible to find a solution to the homogeneous, constant coefficient equation, $h$, with $0 \in \S(h)$, that is as close as we wish to $u$ in an inner domain.

\begin{lemma}\label{approx} Let $u \in W^{1,p}(B_1)$ be a weak solution to \eqref{Eq}, where $a$ satisfies the structural conditions \eqref{Hyp a} and \eqref{Dini}, normalized as to $\intav{B_1} |u|^p dX \le 1$.  Then, given $\delta>0$, there exists a constant $\varepsilon > 0$, depending only on $\delta,  n, p, \lambda, \Lambda,  \tilde{\Lambda}$, and $\omega$, such that if 
\begin{eqnarray}
	|Du(0)| &\le& \varepsilon \label{small cond CL 1}\\
	\|\mu\|_{{L^q}}  &\le& \varepsilon \label{small cond CL 2}\\
	|a(X,\xi) - a(0, \xi)| & \le& \varepsilon \cdot |\xi|^{p-1} , \label{small cond CL 3}
\end{eqnarray}
then there exists a function $h \in W^{1,p}(B_{1/2})$, satisfying
 \begin{eqnarray}
	 -\div (\overline{a}(D h)) &=& 0,  ~ B_{1/2} \label{h -CL 1}, \\ 
	 D h(0) &=& 0 , \label{h -CL 2}
\end{eqnarray}
for some constant coefficient vector field $\overline{a}$, satisfying \eqref{Hyp a}, such that
\begin{equation}\label{thesis - CL}
	  \intav{B_{1/2}} |u - h|^p dX \le \delta^p.
\end{equation}
\end{lemma}

\begin{proof}
Let us assume, for the purpose of contradiction, that the thesis of the Lemma fails. If so, there would exist a $\delta_0 > 0$ and a sequence $
 u_k \in W^{1,p}(B_1),$ with
 \begin{equation}\label{prof comp CM eq 00}
     \intav{B_1} |u_k(X)|^p dX \le 1
 \end{equation}
 for all $k\ge1$, sequences $a_k$ and  $\mu_k$, with  
\begin{equation}\label{prof comp CM eq 01}
           - \div ( a_k (X, Du_k) ) = \mu_k \text{ in } B_1,
\end{equation}
where $a_k$ satisfies \eqref{Hyp a},  and \eqref{Dini} and
\begin{eqnarray}
	|Du_k(0)| &=& \text{o}(1) \label{prof comp CM eq 01.1}\\
	\|\mu_k\|_{{L^q}}  &=& \text{o}(1)  \label{prof comp CM eq 01.2}\\
	|a_k(X,\xi) - a_k(0, \xi)| & =& \text{o}(1)  \cdot  |\xi|^{p-1}, \label{prof comp CM eq 01.3}
\end{eqnarray}
for $\text{o}(1) \to 0$, as $k \to 0$; however
\begin{equation}\label{prof comp CM eq 03}
            \intav{B_{1/2}} |u_k(X) - h(X)|^p dX \ge \delta_0, \quad \forall k \ge 1,
\end{equation}
for any solution $h$ to a homogeneous, constant coefficient
equation as in \eqref{h -CL 1}, in $B_{1/2}$, satisfying $D h(0) = 0$ as entitled in \eqref{h -CL 2}. From the normalization assumption  \eqref{prof comp CM eq 00}, equation \eqref{prof comp CM eq 01} and natural bounds coming from the structural assumption \eqref{Hyp a} and \eqref{prof comp CM eq 01.2}, it follows through Caccioppoli's type energy estimates that
$$
    \int_{B_{1/2}} |D u_k|^p dX \le C,
$$
for all $k \ge 1$. Thus, by compactness embedding and classical truncation arguments, there exists a function $u \in W^{1,p}(B_{1/2})$ for which, up to a subsequence,
\begin{eqnarray}
    u_k &\rightharpoonup&  u_\infty \text{ in } W^{1,p}(B_{1/2}) \label{prof comp CM eq 04}  \\
    u_k &\to & u_\infty \text{ in } L^p(B_{1/2}) \label{prof comp CM eq 04.1}  \\
    D u_k(X) &\to& D u_\infty(X) \text{ for a.e. } X  \in B_{1/2} \label{ae conv grad},
\end{eqnarray}
Also, by Ascoli Theorem, there exists a subsequence under which
$a_{k_j}(0, \cdot) \to \overline{a}(0, \cdot)$ locally uniformly.
Hence, from \eqref{prof comp CM eq 01.3} for any $X \in B_{1/2}$ and $\xi \in B_R$, for arbitrary $R >0$ fixed, there holds
\begin{equation}\label{prof comp CM eq 05}
    |a_{k_j}(X, \xi) - \overline{a}(0, \xi)| \le |a_{k_j}(X, \xi) - a_{k_j}(0, \xi)| + | a_{k_j}(0, \xi) -
    \overline{a}(0, \xi)| = \text{o}(1).
\end{equation}
That is, for we have verified
\begin{equation}\label{prof comp CM eq 05.5}
	a_{k_j}(X, \xi) \to \overline{a}(0, \xi) \text{ locally
uniformly in } B_{1/2} \times \mathbb{R}^n.
\end{equation}
Given a test function $\phi \in W^{1,p}_0(B_{1/2})$, in view of \eqref{prof comp CM eq 01.2},
\eqref{prof comp CM eq 04}, \eqref{ae conv grad}, \eqref{prof comp CM eq 05} and \eqref{prof comp CM eq 05.5} we have
$$
    \begin{array}{lll}
        \displaystyle \int_{B_{1/2}} \overline{a}(0, D u_\infty) \cdot D \phi dX &=& \displaystyle \int_{B_{1/2}} a_k(X, D u_k) \cdot D \phi dX +\text{o}(1) \\
        & = & \text{o}(1),
    \end{array}
$$
as $k \to \infty$. Since $\phi$ was arbitrary, we conclude $u$ is a  solution to a constant coefficient equation in $B_{1/2}$, i.e.,
\begin{equation}\label{prof comp CM eq 06}
	-\div \left ( \overline{a}(0, Du_\infty)  \right ) = 0, \text{ in } B_{1/2}.
\end{equation}
Also, by the $C^{1}$ regularity theory for varying coefficient equations, Theorem \ref{C1 est},  jointly with the asymptotically flat assumption \eqref{prof comp CM eq 01.1}, yield the pointwise value
\begin{equation}\label{prof comp CM eq 07}
	D u_\infty(0) =0.
\end{equation} 
Finally, confronting the conclusions obtained in \eqref{prof comp CM eq 06} and \eqref{prof comp CM eq 07} with \eqref{prof comp CM eq 04.1} and  \eqref{prof comp CM eq 03}, we reach a contradiction for $k \gg 1$. The Lemma is proven.
\end{proof}

\section{Iterative flatness improvement} \label{section flat improv}

In the previous Section we have established an approximation result that provides a {\it flat}, $C^{1,\alpha_M}$ smooth function near a generic solution $u$ to Equation \eqref{Eq Introd}, provided $Du(0)$ and $\mu$ are near zero, and $a(X,\xi) \sim a(0,\xi)$. In this Section we prove essentially a step-one, discrete version of Theorem \ref{main NH}.  

\begin{lemma}\label{key lemma} Let $u \in W^{1,p}(B_1)$ be a weak solution to \eqref{Eq}, where $a$ satisfies the structural conditions \eqref{Hyp a} and \eqref{Dini} and assume $\intav{B_1} |u|^p dX \le 1$. Then, given $0< \alpha < \alpha_M$, there exist constants $0 < \varepsilon_0 < 1$, $0 < \varrho < \frac{1}{2}$, depending only upon $n, p, \lambda, \Lambda,   \tilde{\Lambda},  \omega$ and $\alpha$, such that if
\begin{eqnarray}
	|Du(0)| &\le& \varepsilon_0 \label{small cond key lemma 1}\\
	\|\mu\|_{{L^q}}  &\le& \varepsilon_0 \label{small cond key lemma 2}\\
	|a(X,\xi) - a(0, \xi)| & \le& \varepsilon_0 \cdot  | \xi|^{p-1}, \label{small cond key lemma 3}
\end{eqnarray}
then, we can find a universally bounded real constant $\tau \in \mathbb{R}$, i.e., $|\tau|< C(n, p, \lambda, \Lambda)$, such that
\begin{equation}\label{thesis key lemma}
	\intav{B_\varrho} | u(X) - \tau|^p dX \le \varrho^{p(1+\alpha)}.
\end{equation}
\end{lemma}

\begin{proof}
For $\delta > 0$ to be chosen later, let $h$ be a solution to a constant coefficient equation 
$$
	-\div \left ( \overline{a}(Dh) \right ) = 0, \ B_{1/2}
$$
satisfying $0 \in \S(h)$ that is $\delta$-close to $u$ in the $L^p$-norm. The existence of such a function has been granted by Lemma \ref{approx}. From $C^{1,\alpha_M}$ regularity theory for constant coefficient equations,  there exists a constant $C$ depending only on universal parameters, such that
$$
    |h(X) - h(0)| \le C|X|^{1+\alpha_M}.
$$
Since $\|h\|_{L^p} \le C$, by $L^\infty$ bounds,
$$
    |h(0)| \le C.
$$
We now estimate, for $\varrho >0$ to be adjusted {\it a posteriori},
$$
    \begin{array}{lll}
        \displaystyle \intav{B_{\varrho}} |u(X) - h(0)|^p dX & \le &   2^{p-1} \left (\displaystyle \intav{B_{\varrho}} |u(X) - h(X)|^p dX +   \intav{B_{\varrho}} |h(X) - h(0)|^p dX  \right ) \\
        & \le & 2^{p-1} \delta^p \varrho^{-n} + 2^{p-1} C\varrho^{p(1+\alpha_M)}.
    \end{array}
$$
Since $0 < \alpha < \alpha_M$, it is possible to select $\varrho$ small enough as to assure
$$
     2^{p-1} C\varrho^{p(1+\alpha)} \le \dfrac{1}{2} \varrho^{p(1+\alpha_M)}.
$$
Once selected $\varrho$, as indicated above, we set
$$
    \delta := \dfrac{1}{2} \varrho^{\frac{n}{p} + 1+ \alpha},
$$
which determines the smallness condition $\varepsilon_0$, in the statement of this Lemma, through the singular approximation Lemma \ref{approx}. The proof is concluded. 
\end{proof}

\bigskip

%%%
%%%
%%%

The strategy to be followed in the next Sections in order to deliver a conclusive proof for Theorem \ref{main NH} is to iterate Lemma \ref{key lemma} for dyadic balls with appropriate geometric decaying radii. However, initially we need to show that the proof of Theorem \ref{main NH} can be reduced to the smallness assumptions within the statement of Lemma \ref{key lemma}. This latter task is the objective of our next Section.

\section{The nature of scaling and reduction to  smallness regime} \label{section red}

We start off this Section by commenting on the scaling nature of Equation \eqref{Eq Introd}. Namely, if $u \in W^{1,p}$ solves  Equation \eqref{Eq Introd}, say in $B_1$, then for $X_0 \in B_1$, $0<\zeta< 1-|X_0|$ and $\kappa > 0$, the re-scaled function $v \in W^{1,p}(B_1)$, defined by
$$
	v(X):=  \dfrac{u(X_0 + \zeta X)}{\kappa},
$$
satisfies,
$$
	-\div \left (  a(X_0 + \zeta X, Dv) \right ) = \tilde{\mu},
$$
for some measurable function $\tilde{\mu}$. In view of \eqref{Hyp a}, we estimate:
\begin{equation}\label{scaling est prel}
	\|\tilde{\mu}\|_{L^q(B_1)} \le   \dfrac{\zeta^p}{\kappa^{p-1}} \cdot \zeta^{-\frac{n}{q}} \|\mu\|_{L^q(B_1)}.
\end{equation}
Also, it readily follows from change of variables that
\begin{equation}\label{scaling est 02}
	 \intav{B_1} |v(X)|^p \le \dfrac{1}{\kappa^p \cdot \zeta^n} \intav{B_1} |u(X)|^p dX.
\end{equation}
Finally, if we define 
$$
	\tilde{a}(X, \xi) := a(X_0 + \zeta X, \xi),
$$	
the $\zeta$-variable expansion together with the continuity assumption on the coefficients of $a$, namely assumption \eqref{Dini}, gives
\begin{equation}\label{scaling est 03}
	\left | \tilde{a}(X, \xi) - a(0, \xi) \right | \le \tilde{\Lambda} \omega(\zeta |X|) |\xi|^{p-1}
\end{equation}

These simple facts show that in order to establish the proof of Theorem \ref{main NH}, we can start off under the assumptions of Lemma \ref{key lemma}. That is, we have the following logistic device:

\begin{proposition}\label{prop red} Assume we can prove, under the assumptions of Lemma \ref{key lemma}, that if $0 \in \S(v)$, then
\begin{equation}\label{proved under hyp key lemma}
	|v(X) - v(0)| \le C |X |^{1+\alpha},
\end{equation}
 for a constant $C$ depending only on  $n, p, \lambda, \Lambda,   \tilde{\Lambda},  \omega$ and $\alpha$. Then, for any function $u \in W^{1,p}(\Omega)$, solution to Equation \eqref{Eq Introd}, with $a$ satisfying \eqref{Hyp a}, \eqref{Dini} and $\mu$ under condition \eqref{measure cond}, there holds,
\begin{equation}\label{proved general}
	|u(Y) - u(Y_0)| \le \tilde{C} |Y-Y_0|^{1+\alpha},
\end{equation}
for any point $Y_0 \in \S(u)$, where $\tilde{C}>0$ is another constant that depends only on $\|u\|_{L^p(\Omega)}$, $\dist(Y_0, \partial \Omega)$,  $n, p, \lambda, \Lambda,   \tilde{\Lambda},  \omega$, $\|\mu\|_{L^q(\Omega)}$ and $\alpha$.
\end{proposition}
\begin{proof} Given a generic function $u \in W^{1,p}$, satisfying Equation \eqref{Eq Introd}, with $a$ obeying \eqref{Hyp a}, \eqref{Dini} and $\mu$ satisfying \eqref{measure cond}, we define the re-scaled function
\begin{equation}\label{rescaled function}
	v(X):=  \dfrac{u(Y_0 + \zeta X)}{\kappa},
\end{equation}
for an arbitrary interior singular point $Y_0 \in \S(u)$.
In view of \eqref{scaling est 02} and \eqref{scaling est 03}, we are led to choose
$$
	\zeta := \min\left \{ 1, \ \frac{1}{2} \dist(Y_0, \partial \Omega), \ \sqrt[p- \frac{n}{q}]{\frac{\varepsilon_0}{\|\mu\|_{L^{q(\Omega)}}}}, \ \omega^{-1}(\tilde{\Lambda}^{-1} \varepsilon_0) \right \},
$$
where $\varepsilon_0$ is the universal constant from Lemma \ref{key lemma}. In the sequel, we take
$$ 
	\kappa := \max \left \{1, \ \sqrt[p]{\dfrac{1}{\zeta^n} \cdot \intav{B_1} |u(X)|^p dX} \right \}.
$$
Under these selections, the re-scaled function $v \in W^{1,p}(B_1)$ fulfills the smallness assumptions of Lemma \ref{key lemma}, and $0 \in \S(v)$.  Applying the presumed proven estimate \eqref{proved under hyp key lemma} to $v$, we find
$$
		\left | u(Y) - u(Y_0) \right | \le C\cdot  \kappa \cdot \left ( \frac{1}{\zeta} \right )^{1+\alpha} |Y-Y_0|^{1+\alpha},
$$
which is precisely the aimed conclusion \eqref{proved general}. 
\end{proof}

\section{Conclusion of the proof} \label{section final proof}

In this Section we finish up the proof of Theorem \ref{main NH}. Initially, it follows from the conclusion of Proposition \ref{prop red}, that we can assume $0\in \S(u)$ and $u$, $a$ and $\mu$ are under the smallness assumptions requested within the statement of Lemma \ref{key lemma}, i.e., 
\begin{eqnarray}
	0 &\in& \S(u) \label{small cond final proof 1.5}\\
	\displaystyle \intav{B_1} |u(X)|^p dX  &\le& 1 \label{small cond final proof 1}\\
	\|\mu\|_{{L^q}}  &\le& \varepsilon_0 \label{small cond final proof 2}\\
	|a(X,\xi) - a(0, \xi)| & \le& \varepsilon_0 \cdot |\xi|^{p-1}, \label{small cond final proof 3}
\end{eqnarray}
where $\varepsilon_0$ is the universal number from Lemma \ref{key lemma}. Under such conditions, we will show that there exists a  sequence of real numbers $\tau_k \to u(0)$, such that

\begin{equation}\label{proof main 01}
 	\intav{B_{\varrho^k}} |u(X) - \tau_k|^p dX \le \varrho^{k \cdot p(1+\alpha)},
\end{equation} 
where $\varrho>0$ is the small universal radius from Lemma \ref{key lemma} and $\alpha$ is a fixed exponent within the range
\begin{equation}\label{proof main range alpha}
	\alpha \in \left (0, \alpha_M \right ) \cap \left ( 0, \frac{q-n}{q(p-1)} \right ].
\end{equation}
We argue by finite induction. The case $k=1$ is precisely the thesis of Lemma \ref{key lemma}. Suppose we have verified \eqref{proof main 01} for $k$. Define the re-normalized function $v \colon B_1 \to \mathbb{R}$ by
\begin{equation}\label{proof main 02}
	v(X) := \dfrac{u(\varrho^k X) - \tau_k}{\varrho^{k \cdot (1+\alpha)}}.
\end{equation}
If we label
\begin{equation}\label{proof main 02.1}
	-\div \left ( a(\varrho^kX, Dv) \right ) =: \mu_k,
\end{equation}
it follows from the structural scaling of the equation and change of variable arguments, as in \eqref{scaling est prel}  that
\begin{equation}\label{proof main 03}
	\begin{array}{lll}
		\|\mu_k\|^q_{L^q(B_1)} &\le& \varrho^{k[p - (1+\alpha)(p-1)]q} \cdot  \varrho^{-kn} \|\mu\|^q_{L^q(B_{\varrho^k})} \\
		&\le& \varepsilon_0^q \cdot\varrho^{kq[p - (1+\alpha)(p-1) - \frac{n}{q}]} . 
	\end{array}
\end{equation}
From the sharp selection of the exponent in \eqref{proof main range alpha} , namely
$$
	\alpha \le \dfrac{q-n}{(p-1)q}
$$ 
and the estimate obtained in \eqref{proof main 03}, we deduce,
\begin{equation}\label{proof main 04}
	\|\mu_k\|_{{L^q(B_1)}} \le \varepsilon_0.
\end{equation}
Also, easily we check that
\begin{equation}\label{proof main 05}
	\begin{array}{lll}
		|a(\varrho^k X,\xi) - a(0, \xi)| &\le&  \tilde{\Lambda} \omega(\varrho^k |X|) \cdot  |\xi|^{{p-1}} \\
		&\le & \varepsilon_0  \cdot  |\xi|^{{p-1}} .
	 \end{array}
\end{equation}
We have shown that $v$ is under the hypotheses of Lemma \ref{key lemma}, which assures the existence of a universally bounded real constant $\tilde{\tau}$ such that
\begin{equation}\label{proof main 06}
	\intav{B_\varrho} | v(X) - \tilde{\tau}|^p dX \le \varrho^{p(1+\alpha)}.
\end{equation}
If we define 
\begin{equation}\label{proof main 07}
	\tau_{k+1} := \tau_k + \varrho^{k(1+\alpha)} \tilde{\tau}
\end{equation}
and rescale estimate \eqref{proof main 06} back, we conclude the proof of \eqref{proof main 01}. It further follows from \eqref{proof main 07} that,
\begin{equation}\label{proof main 08}
	|\tau_{k+1} - \tau_k| \le C \varrho^{k(1+\alpha)},
\end{equation}
for a universal constant $C>0$. In particular $\{\tau_k\}_{k\ge 1}$ is a Cauchy sequence. From \eqref{proof main 01}, we deduce
$$
	\lim\limits_{k\to \infty} \tau_k = u(0).
$$
Yet a consequence of estimate \eqref{proof main 08} is the following convergence rate control
\begin{equation}\label{proof main 08.1}
	\begin{array}{lll}
		| \tau_k - u(0)| &\le& \displaystyle  C \sum\limits_{j={k}}^\infty \varrho^{j(1+\alpha)}  \\
		&\le& \displaystyle \frac{C}{1-\varrho} \varrho^{k(1+\alpha)} 
	\end{array}
\end{equation}
Finally, given any $0<r \ll 1$, let $k$ be the natural number that satisfies
\begin{equation}\label{proof main 09}
	 \varrho^{k+1} \le r < \varrho^k. 
\end{equation}
We estimate
\begin{equation}\label{proof main 10}
	\begin{array}{lll}
		\displaystyle \intav{B_r} |u(X) - u(0)|^p dX &\le&   \left ( \frac{\varrho^k}{r} \right )^n  \displaystyle\intav{B_{\varrho^k}} |u(X) - u(0)|^p dX \\
		&\le&   \frac{2^{p-1}}{\varrho^n} \displaystyle \left ( \intav{B_{\varrho^k}} |u(X) -\tau_k|^p dX + |\tau_k - u(0)|^p \right )\\
		&\le &   \frac{2^{p-1}}{\varrho^n} \cdot \left (1 +\left ( \frac{C}{1-\varrho} \right )^p \right ) \varrho^{k\cdot p(1+\alpha)} \\
		& \le & \left [ \frac{2^{p-1}}{\varrho^n} \cdot \left (1 +\left ( \frac{C}{1-\varrho} \right )^p \right ) \cdot \frac{1}{\varrho^{p(1+\alpha)}} \right ] \cdot r^{p(1+\alpha)} \\
		& = & \tilde{C} \cdot r^{p(1+\alpha)},
	 \end{array}
\end{equation}
for a constant $\tilde{C} > 0$ that depends only upon universal parameters. The $C^{1,\alpha}$ regularity of $u$ at $0$ follows now by standard arguments.  \qed

%%%%%%%%%%%%%%%%%%%%%
%%%%%%%%%%%%%%%%%%%%%
%%%%%%%%%%%%%%%%%%%%%
%%%%%%%%%%%%%%%%%%%%%
\bibliographystyle{amsplain, amsalpha}

\bigskip

\noindent \textsc{Eduardo V. Teixeira} \\
\noindent Universidade Federal do Cear\'a \\
\noindent Departamento de Matem\'atica \\
\noindent Campus do Pici - Bloco 914, \\
\noindent Fortaleza, CE - Brazil 60.455-760 \\
 \noindent \texttt{teixeira@mat.ufc.br}

\end{document}